\documentclass[11pt,a4paper,twoside]{article}
\usepackage{a4wide}

\usepackage[utf8]{inputenc}
\usepackage[T1]{fontenc}
\usepackage{amsmath,amsthm,amsfonts,bbm}
\usepackage{graphicx}

\newtheorem{theorem}{Theorem}
\newtheorem{lemma}[theorem]{Lemma}
\newtheorem{remark}[theorem]{Remark}

\begin{document}

\title{A note on the space-time variational formulation for the wave equation with source term in $L^2(Q)$}

\author{Marco Zank}
\date{
      Fakultät für Mathematik, Universität Wien, Oskar--Morgenstern-Platz 1, 1090 Wien, Austria, \\[1mm] 
      {\tt marco.zank@univie.ac.at}
}

\maketitle

%------------------------------------------------------------------------------
\begin{abstract}
    We derive a variational formulation for the scalar wave equation in the second-order formulation on bounded Lipschitz domains and homogeneous initial conditions. We investigate a variational framework in a bounded space-time cylinder~$Q$ with a new solution space and the test space $L^2(Q)$ for source terms in $L^2(Q)$. Using existence and uniqueness results in $H^1(Q)$, we prove that this variational setting fits the inf-sup theory, including an isomorphism as solution operator. Moreover, we show that the new solution space is not a subspace of $H^2(Q)$. This new uniqueness and solvability result is not only crucial for discretizations using space-time methods, including least-squares approaches, but also important for regularity results and the analysis of related space-time boundary integral equations, which form the basis for space-time boundary element methods.

    \bigskip

\end{abstract}

%------------------------------------------------------------------------------
\section{Introduction}

The solutions to hyperbolic evolution problems are essential in many physical areas. As these solutions are often accessible only as approximations, different numerical schemes and their analysis are available. The most popular approaches are based on time-stepping schemes and spatial discretizations, e.g., finite element methods or finite difference methods. As time-stepping approaches are of pointwise-in-time-type, these methods are not based on a variational formulation in both, the spatial variable $x$ and the temporal variable $t$. In recent years, approaches using also a variational setting with respect to $t$, have been developed, which we call space-time methods, see \cite{Zlotnik1994} for temporal first-order and second-order formulations. Not only for these methods but also in general, the underlying space-time functional framework used for hyperbolic evolution equations is crucial for, e.g., regularity results or the analysis of related boundary integral equations or any conforming discretization method. The model problem is the scalar wave equation in second-order formulation
\begin{equation} \label{zank:einf:WellePDG}
    \left. \begin{array}{rcll}
    c^{-2} \partial_{tt} u - \Delta_x u & = & f  & \mbox{in} \;
    Q = \Omega \times (0,T), \\[1mm]
    u & = & 0 & \mbox{on} \; \Sigma = \partial\Omega \times [0,T],
    \\[1mm] 
    \partial_t u(\cdot,0) = u(\cdot,0) & = & 0 & \mbox{on} \; \Omega
    \end{array} \right\}
\end{equation}
with a given source term $f$ and the lateral boundary $\Sigma = \partial \Omega \times [0,T]$ of the space-time cylinder $Q = \Omega \times (0,T)$, where $\Omega  \subset \mathbb{R}^d$, $d \in \{1,2,3\}$, is a bounded Lipschitz domain, the constant wave speed $c$ is set to 1 due to rescaling and $T>0$. Many space-time approaches introduce additional unknowns, which represent first-order derivatives, see \cite{Zlotnik1994, FerrariPerugiaZampa2025, FuehrerGonzalezKarkulik2023, Gomez2025, KoecherBause2014}. Other space-time variational settings in \cite{BignardiMoiola2025} require
a star-shaped spatial domain $\Omega$ with additional assumptions on the boundary,
and besides this, the analysis of the ultraweak variational formulations of \cite{HenningPalittaSimonciniUrban2022} are strongly based on isomorphisms between function spaces, which are pointwise in time. Alternatively, we recall the variational setting in $H^1(Q)$, analyzed in \cite{Ladyzhenskaya1985, SteinbachZank2019Stabilized, Zlotnik1994}, by introducing the space-time Sobolev spaces $H^{1,1}_{0;0,}(Q)$, satisfying the homogeneous initial condition for $t=0$ and the homogeneous Dirichlet boundary condition on $\Sigma$, and $H^{1,1}_{0;,0}(Q)$, fulfilling the homogeneous terminal condition for $t=T$ and the homogeneous Dirichlet boundary condition on $\Sigma$, see Section~\ref{zank:sec:FS}. With these spaces, we set the continuous bilinear form
\begin{equation*}
  a_{H^1}(\cdot,\cdot) \colon \, H^{1,1}_{0;0,}(Q) \times H^{1,1}_{0;,0}(Q) \to \mathbb{R}
\end{equation*}  
by
\begin{equation} \label{zank:einf:aH1}
  a_{H^1}(v,w) = \int_0^T \int_\Omega \left( - \partial_t v(x,t) \partial_t w(x,t)  + \nabla_x v(x,t) \cdot \nabla_x w(x,t) \right) \mathrm dx \mathrm dt.
\end{equation}
Then, the space-time variational formulation of finding $u \in H^{1,1}_{0;0,}(Q)$ such that
\begin{equation} \label{zank:einf:VFH1}
  \forall w \in H^{1,1}_{0;,0}(Q): \quad a_{H^1}(u,w) = \int_0^T \int_\Omega f(x,t) w(x,t) \mathrm dx \mathrm dt
\end{equation}
has a unique solution for each source term $f \in L^2(Q)$, fulfilling the stability estimate
\begin{equation} \label{zank:einf:stab}
  \| u \|_{H^{1,1}_{0;0,}(Q)} = \sqrt{ \| \partial_t u \|_{L^2(Q)}^2 + \| \nabla_x u \|_{L^2(Q)^d}^2 } \leq \frac{1}{\sqrt{2}} T \| f \|_{L^2(Q)}
\end{equation}
and the regularities
    $u \in  C([0,T]; H^1_0(\Omega))$ and $\partial_{t} u \in C([0,T]; L^2(\Omega)).$
The variational formulation~\eqref{zank:einf:VFH1} is successfully used in the discretization schemes 
\cite{Fraschini2024, LoescherSteinbachZankDD, SteinbachZank2019Stabilized, Zank2021ECCOMAS, Zank2021Enumath, Zlotnik1994}.
However, a natural question is whether the bilinear form $a_{H^1}(\cdot,\cdot)$ in \eqref{zank:einf:aH1} fulfills the inf-sup theorem~\cite[Thm.~25.9]{ErnGuermondII2021}.
It turns out that the solution operator~$\mathcal L_{H^1} \colon \, L^2(Q) \to H^{1,1}_{0;0,}(Q)$ to the space-time variational formulation~\eqref{zank:einf:VFH1}, defined by $\mathcal L_{H^1}f = u$, is not surjective. Therefore, $\mathcal L_{H^1}$ is not an isomorphism and the bilinear form  $a_{H^1}(\cdot,\cdot)$ in \eqref{zank:einf:aH1} does not fit to the inf-sup theorem~\cite[Thm.~25.9]{ErnGuermondII2021}. We summarize this result in the following theorem, see Section~\ref{zank:sec:FS} for the notation.

\begin{theorem}{\cite[Thm.~1.1]{SteinbachZank2022VF}}
  \label{zank:einf:thm:InfSupH1}
  There does not exist a constant $C>0$ such that each source term
  $f \in L^2(Q)$ and the corresponding solution $u \in H^{1,1}_{0;0,}(Q)$
  of \eqref{zank:einf:VFH1} satisfy
  \begin{equation*}
    \| u \|_{H^{1,1}_{0;0,}(Q)} \, \leq \, C \, \| f \|_{[H^{1,1}_{0;\,,0}(Q)]'} .
  \end{equation*}
  In particular, the inf-sup condition
  \begin{equation*}
    \forall v \in H^{1,1}_{0;0,}(Q): \quad c_S \, \| v \|_{H^{1,1}_{0;0,}(Q)} \, \leq \,
    \sup_{0 \neq w \in H^{1,1}_{0;\,,0}(Q)}
    \frac{| a_{H^1}(v,w) |}{\| w \|_{H^{1,1}_{0;,0}(Q)}}
    \end{equation*}
    with a constant $c_S>0$ does not hold true.
\end{theorem}
One possibility to derive a variational setting, which satisfies the inf-sup theorem~\cite[Thm.~25.9]{ErnGuermondII2021}, is to enlarge the solution space $H^{1,1}_{0;0,}(Q)$ of the variational formulation~\eqref{zank:einf:VFH1} to a Hilbert space $\mathcal H_{0;0,}(Q) \supset H^{1,1}_{0;0,}(Q)$
and to enlarge also the space of the source term~$f$ to $[H^{1,1}_{0;,0}(Q)]'$,
leading to an isomorphism~$\mathcal L_{\mathcal H_{0;0,}} \colon \, [H^{1,1}_{0;,0}(Q)]' \to \mathcal H_{0;0,}(Q)$ as solution operator, which is the main idea of \cite{SteinbachZank2022VF}. Another possibility is to shrink the solution space $H^{1,1}_{0;0,}(Q)$ of the variational formulation~\eqref{zank:einf:VFH1} to a Hilbert space $H^1_{0;0,}(Q;\square) \subset H^{1,1}_{0;0,}(Q)$, resulting in an isomorphism~$\mathcal L_{H^1_{0;0,}(Q;\square)} \colon \, L^2(Q) \to H^1_{0;0,}(Q;\square)$ as solution operator. The latter approach is the content of this work. We introduce the new solution space $H^1_{0;0,}(Q;\square) \subset H^{1,1}_{0;0,}(Q)$ and a related bilinear form $a(\cdot, \cdot) \colon \, H^1_{0;0,}(Q;\square) \times L^2(Q) \to \mathbb{R}$. Using this, we prove that $a(\cdot, \cdot)$ fits to the framework of the inf-sup theorem~\cite[Thm.~25.9]{ErnGuermondII2021}, which includes the unique solvability of the related space-time variational setting to the wave equation~\eqref{zank:einf:WellePDG} for each source term~$f \in L^2(Q)$. Further, we prove that the new solution space $H^1_{0;0,}(Q;\square)$ is not a subspace of $H^2(0,T;L^2(\Omega))$.

The rest of the paper is organized as follows: In Section~\ref{zank:sec:FS}, we summarize notations of function spaces. Section~\ref{zank:sec:VF} is the main part of this work, where we introduce the new space-time solution space and the new space-time variational formulation. We end this section by proving unique solvability with the help of the inf-sup theorem~\cite[Thm.~25.9]{ErnGuermondII2021}. In Section~\ref{zank:sec:Zum}, we give some conclusions.

%------------------------------------------------------------------------------
\section{Functions spaces} \label{zank:sec:FS}

In this section, we introduce the used function spaces.

For a bounded Lipschitz domain $A \subset \mathbb{R}^n$, $n \in \mathbb{N}$, we denote by $C_0^\infty(A)$ the set of all real-valued functions with compact support in $A$, which are infinitely differentiable. Further, $\mathcal D(A)$ is the space of test functions when the set $C_0^\infty(A)$ is endowed with the standard locally convex topology. Further, we denote by $L^2(A)$ the usual Lebesgue space of real-valued functions with inner product $\langle \cdot, \cdot \rangle_{L^2(A)}$ and induced norm $\| \cdot \|_{L^2(A)}$. Similarly, we use the classical Sobolev spaces $H^k(A)$, $k \in \mathbb{N}$, and the subspace $H^1_0(A)$, satisfying homogeneous Dirichlet conditions, with dual space $[H^1_0(A)]'$. We extend these spaces to vector-valued versions, e.g., $L^2(A)^m$ for $m \in \mathbb{N}$. 
Further, for $A=(0,T)$, we consider the spaces $H^1_{0,}(0,T)$, satisfying the initial condition $v(0)=0$, and $H^1_{,0}(0,T)$, satisfying the terminal condition $v(T)=0$. Again, for $A=(0,T)$, we introduce the extensions to Bochner spaces $L^2(0,T; X)$ and $H^k(0,T; X)$, $k \in \mathbb{N}$, with a real Hilbert space $X$. Analogously, we use the space $C([0,T]; X)$ of continuous functions.
Moreover, for $A=Q$, we consider the Lebesgue space $L^2(Q)$, the Sobolev space $H^1(Q)$ and its closed subspaces
\begin{align*}
  H^{1,1}_{0;0,}(Q) =& H^1_{0,}(0,T;L^2(\Omega)) \cap L^2(0,T; H^1_0(\Omega)) \subset H^1(Q), \\
  H^{1,1}_{0;,0}(Q) =& H^1_{,0}(0,T;L^2(\Omega)) \cap L^2(0,T; H^1_0(\Omega)) \subset H^1(Q),
\end{align*}
which are endowed with the Hilbertian norm
\begin{equation*}
   \|w \|_{H^{1,1}_{0;0,}(Q)} =  \|w \|_{H^{1,1}_{0;,0}(Q)} = \sqrt{  \|\partial_t w \|_{L^2(Q)}^2 +  \|\nabla_x w \|_{L^2(Q)^d}^2 },
\end{equation*}
due to the Poincar\'{e} inequality. Finally, $X'$ is the dual space of a real Hilbert space~$X$.

%------------------------------------------------------------------------------
\section{New variational formulation for source terms in $L^2(Q)$} \label{zank:sec:VF}

In this section, we prove the main theorems for the wave equation~\eqref{zank:einf:WellePDG}.

The variational formulation~\eqref{zank:einf:VFH1} leads to
\begin{equation*}
    \forall \varphi \in \mathcal D(Q): \quad \langle {u}, {\partial_{tt} \varphi} \rangle_{L^2(Q)} - \langle {u}, {\Delta_x \varphi} \rangle_{L^2(Q)} = \langle {f}, {\varphi} \rangle_{L^2(Q)}.
\end{equation*}
In other words, the distributional wave operator $\square = \partial_{tt} -\Delta_x$ applied to the solution~$u\in H^{1,1}_{0;0,}(Q)$ of the variational formulation~\eqref{zank:einf:VFH1} is given by the $L^2(Q)$ function~$f$, i.e., 
    $\square u = f \in L^2(Q).$
In particular, the solution~$u\in H^{1,1}_{0;0,}(Q)$ of variational formulation~\eqref{zank:einf:VFH1} fulfills $u \in H^1_{0;0,}(Q;\square)$. Here, we define the space
\begin{equation*}
  H^1_{0;0,}(Q;\square) = \{ v\in H^{1,1}_{0;0,}(Q):\, \square v\in L^2(Q)\, \text{ and } \, \partial_t v(\cdot,0)=0\},
\end{equation*}
where the condition $\partial_t v(\cdot,0)=0$ is well-defined due to $\partial_{t} v \in C([0,T]; [H^1_0(\Omega)]').$
The space $H^1_{0;0,}(Q;\square)$ is a Hilbert space with the inner product
\begin{equation*}
	\langle {v}, {w} \rangle_{H^1_{0;0,}(Q;\square)}= \langle {\partial_t v}, { \partial_t w} \rangle_{L^2(Q)} + \langle {\nabla_x v}, { \nabla_x w} \rangle_{L^2(Q)^d} + \langle {\square v}, {\square w} \rangle_{L^2(Q)}
\end{equation*}
and the induced norm $\| v \|_{H^1_{0;0,}(Q;\square)} = \sqrt{\langle v,v \rangle_{H^1_{0;0,}(Q;\square)}}.$

\begin{lemma}
  In $H^1_{0;0,}(Q;\square)$, the norms $\| \cdot \|_{H^1_{0;0,}(Q;\square)}$ and $\| \square(\cdot) \|_{L^2(Q)}$ are equivalent.
\end{lemma}
\begin{proof}
    Let $v \in H^1_{0;0,}(Q;\square)$ be given.
    First, we have that
    $\| \square v \|_{L^2(Q)} \leq \| v \|_{H^1_{0;0,}(Q;\square)}.$

    Second, consider $v$ as the unique solution of the variational formulation~\eqref{zank:einf:VFH1} with source term $f=\square v \in L^2(Q)$ and homogeneous initial and boundary conditions. Thus, the stability estimate~\eqref{zank:einf:stab}
    yields that
    \begin{equation*}
      \| v \|_{H^1_{0;0,}(Q;\square)} \leq \sqrt{T^2/2 + 1} \| \square v \|_{L^2(Q)}.
    \end{equation*}
\end{proof}

In the remainder, we equip the space $H^1_{0;0,}(Q;\square)$ with the inner product $\langle {\square (\cdot)}, {\square (\cdot)} \rangle_{L^2(Q)}$ and the induced norm $\| \square (\cdot)  \|_{L^2(Q)}$. With this, we consider the variational formulation to find $u \in H^1_{0;0,}(Q;\square)$ such that
\begin{equation} \label{zank:VF:VF_0}
    \forall w \in L^2(Q): \quad a(u,w) = \langle {f}, {w} \rangle_{L^2(Q)}
\end{equation}
with a given source term $f \in L^2(Q)$. Here, the bilinear form
\begin{equation*}
  a(\cdot,\cdot)\colon \, H^1_{0;0,}(Q;\square) \times L^2(Q) \to \mathbb{R}
\end{equation*}
is defined by
\begin{equation*}
  a(v,w) = \langle {\square v}, {w} \rangle_{L^2(Q)}
\end{equation*}
for $v \in H^1_{0;0,}(Q;\square) $ and $w \in L^2(Q).$ With this, we state the main theorem.

\begin{theorem}
    The bilinear form $a(\cdot,\cdot)\colon\, H^1_{0;0,}(Q;\square) \times L^2(Q) \to \mathbb{R}$ fulfills the inf-sup theorem~\cite[Thm.~25.9]{ErnGuermondII2021}, i.e., the following three conditions hold true:
    \begin{enumerate}
     \item the continuity or the so-called first condition:
      \begin{equation*}
          \forall v \in H^1_{0;0,}(Q;\square), \, w \in L^2(Q): \quad |a(v,w)| \leq \| \square v \|_{L^2(Q)} \| w \|_{L^2(Q)},
      \end{equation*}
     \item the inf-sup condition or the so-called second condition:
      \begin{equation*}
          \forall v \in H^1_{0;0,}(Q;\square):  \sup_{0 \neq w \in L^2(Q)} \frac{|a(v,w)|}{\| w \|_{L^2(Q)}} = \| \square v \|_{L^2(Q)},
      \end{equation*}
     \item  the so-called third condition:
      \begin{equation*}
          \forall w \in L^2(Q)\setminus \{0\} \colon \exists v \in H^1_{0;0,}(Q;\square) \colon \quad a(v,w) \neq 0.
      \end{equation*}
    \end{enumerate}
    In particular, the variational formulation~\eqref{zank:VF:VF_0} has a unique solution $u \in H^1_{0;0,}(Q;\square)$ fulfilling the stability
    \begin{equation*}
        \| \square u \|_{L^2(Q)} = \| f \|_{L^2(Q)}.
    \end{equation*}
\end{theorem}
\begin{proof}
    The first assertion is proven by the Cauchy--Schwarz inequality.
    
    The second assertion follows from a representation of the $L^2(Q)$ norm, i.e., 
    \begin{equation*}
        \| \square v \|_{L^2(Q)} = \sup_{0 \neq w \in L^2(Q)} \frac{|\langle \square v,w \rangle_{L^2(Q)} |}{\| w \|_{L^2(Q)}} = \sup_{0 \neq w \in L^2(Q)} \frac{|a(v,w)|}{\| w \|_{L^2(Q)}}
    \end{equation*}
    for $v \in H^1_{0;0,}(Q;\square).$
    
    We prove the third assertion: For given $0 \neq w \in L^2(Q)$, let $v\in H^{1,1}_{0;0,}(Q) \subset L^2(Q)$ be the solution of the variational formulation~\eqref{zank:einf:VFH1} for $f=w \in L^2(Q)$. This solution~$v$ fulfills $H^1_{0;0,}(Q;\square)$ and $\square v = w.$ Plugging this into the bilinear form $a(\cdot,\cdot)$ leads to
    \begin{equation*}
	a(v,w) = \langle {\square v}, {w} \rangle_{L^2(Q)} = \| w \|_{L^2(Q)}^2 \neq 0,
    \end{equation*}
    which yields the third assertion.
    
    The unique solvability follows by the inf-sup theorem~\cite[Thm.~25.9]{ErnGuermondII2021}. The stability is proven by
    \begin{equation*}
         \| \square u \|_{L^2(Q)} = \sup_{0 \neq w \in L^2(Q)} \frac{|a(u,w)|}{\| w \|_{L^2(Q)}} = \sup_{0 \neq w \in L^2(Q)} \frac{|\langle f,w \rangle_{L^2(Q)} |}{\| w \|_{L^2(Q)}} =  \| f \|_{L^2(Q)},
    \end{equation*}
    using again the representation of the $L^2(Q)$ norm.
\end{proof}

\begin{remark}
  Inhomogeneous boundary and initial conditions can be treated by superposition. Precise regularity assumptions on the data are left for future work. 
\end{remark}

Last, we show that $H^1_{0;0,}(Q;\square)$ is not a subspace of $H^2(0,T;L^2(\Omega))$. For this purpose, for $j \in \mathbb{N}$, let $\phi_j \in H^1_0(\Omega)$ be the eigenfunctions and $\mu_j> 0$ be the related eigenvalues, fulfilling the eigenvalue problem
\begin{equation*}
    -\Delta_x \phi_j = \mu_j \phi_j \text{ in } \Omega, \quad \phi_{j|\partial \Omega} = 0, \quad \| \phi_j \|_{L^2(\Omega)} = 1.
\end{equation*}
We have that $0 < \mu_1 \leq \mu_2 \leq \dots$ and $\mu_j \to \infty$ as $j \to \infty$. Moreover, Weyl’s asymptotic formula states that there exist constants $c_1>0$ and $c_2>0$ depending on $\Omega \subset \mathbb{R}^d$ and $d$ such that
\begin{equation} \label{zank:VF:Weyl}
    c_1 j^{2/d} \leq \mu_j \leq c_2 j^{2/d}
\end{equation}
holds true for $j \in \mathbb{N}$ sufficiently large, see \cite[Thm.~12.14]{Schmuedgen2012} and its proof. The set $\{ \phi_j : j \in \mathbb{N}\}$ is an orthonormal basis of $L^2(\Omega)$ and an orthogonal basis of $H^1_0(\Omega)$.

\begin{theorem}
  The condition
  \begin{equation*}
    H^1_{0;0,}(Q;\square) \not\subset H^2(0,T;L^2(\Omega))
  \end{equation*}
  holds true, and thus, $H^1_{0;0,}(Q;\square) \not\subset H^2(Q)$.
\end{theorem}
\begin{proof}
    We define
    \begin{equation*}
        u(x,t) = \sum_{j=1}^\infty \phi_j(x) \mu_j^{-d/4-1/2} \int_0^t s \sin (\sqrt{\mu_j} s) \, \mathrm{d}s = \sum_{j=1}^\infty \phi_j(x) \mu_j^{-d/4-1/2} v_j(t)
    \end{equation*}
    for $(x,t) \in Q$ with $v_j(t) = \int_0^t s \sin (\sqrt{\mu_j} s) \, \mathrm{d}s$
    and the partial sums
    \begin{equation*}
        u_M(x,t) = \sum_{j=1}^M \phi_j(x) \mu_j^{-d/4-1/2} \int_0^t s \sin (\sqrt{\mu_j} s) \, \mathrm{d}s = \sum_{j=1}^M \phi_j(x) \mu_j^{-d/4-1/2} v_j(t)
    \end{equation*}
    for $(x,t) \in Q$ with $M \in \mathbb{N}$. First, we prove that $u \in H^1_{0;0,}(Q;\square)$ by using the same Fourier techniques as in \cite[Chap.~IV]{Ladyzhenskaya1985}, where we sketch the main steps only. We calculate $\| v_j \|_{L^2(0,T)}^2$ and $\| \partial_t v_j \|_{L^2(0,T)}^2$. Using Weyl’s asymptotic formula~\eqref{zank:VF:Weyl}, direct calculations yield
    \begin{equation*}
        \| v_j \|_{L^2(0,T)}^2 =  \frac{T^3}{6 \mu _j}+\frac{T^2 \sin \left(2 T \sqrt{\mu _j}\right)}{4 \mu _j^{3/2}}+\frac{T}{2 \mu _j^2} 
        -\frac{5 \sin \left(2 T \sqrt{\mu _j}\right)}{8 \mu _j^{5/2}}+\frac{3 T \cos \left(2 T \sqrt{\mu _j}\right)}{4 \mu _j^2} \in \mathcal O(j^{-2/d})
    \end{equation*}
    and
    \begin{equation*}
        \| \partial_t v_j \|_{L^2(0,T)}^2 =  -\frac{T^2 \sin \left(2 T \sqrt{\mu _j}\right)}{4 \sqrt{\mu _j}}+\frac{\sin \left(2 T \sqrt{\mu _j}\right)}{8 \mu _j^{3/2}}-\frac{T \cos \left(2 T \sqrt{\mu _j}\right)}{4 \mu _j}+\frac{T^3}{6} 
        \in \mathcal O(1)
    \end{equation*}
    as $j \to \infty$. These calculations, Weyl’s asymptotic formula~\eqref{zank:VF:Weyl} and the Fourier representations of the norms give that
    \begin{equation*}
        \| u \|_{L^2(Q)}^2 = \sum_{j=1}^\infty  \mu_j^{-d/2-1} \| v_j \|_{L^2(0,T)}^2 \leq C \sum_{j=1}^\infty j^{-1 -2/d -2/d} < \infty,
    \end{equation*}
    \begin{equation*}
        \| \partial_t u \|_{L^2(Q)}^2 = \sum_{j=1}^\infty  \mu_j^{-d/2-1} \| \partial_t v_j \|_{L^2(0,T)}^2 \leq C \sum_{j=1}^\infty j^{-1 -2/d} < \infty,
    \end{equation*}
    \begin{equation*}
        \| \nabla_x u \|_{L^2(Q)^d}^2 = \sum_{j=1}^\infty  \mu_j^{-d/2-1} \mu_j \| v_j \|_{L^2(0,T)}^2 \leq C \sum_{j=1}^\infty j^{-1 -2/d+2/d-2/d} < \infty
    \end{equation*}
    with a constant $C>0$ depending on $d$, $T$ and $\Omega$. So, we have that $u \in H^1(Q)$. Next, we compute formally that
    \begin{equation*}
        \square u(x,t) = \sum_{j=1}^\infty 2 \phi_j(x) \mu_j^{-d/4-1/2} \sin(\sqrt{\mu_j} t), \quad (x,t) \in Q,
    \end{equation*}
    and further,
    \begin{equation*}
        \| \square u \|_{L^2(Q)}^2 = \sum_{j=1}^\infty \mu_j^{-d/2-1} \left( 2 T-\frac{\sin \left(2 T \sqrt{\mu_j }\right)}{\sqrt{\mu_j }} \right) \leq C \sum_{j=1}^\infty j^{-1 - 2/d} < \infty
    \end{equation*}
    with a constant $C>0$ depending on $d$, $T$ and $\Omega$, where we use Weyl’s asymptotic formula~\eqref{zank:VF:Weyl}. Hence, $u \in H^1(Q)$ and $\square u \in L^2(Q)$. Moreover, as $u_M \in H^1_{0;0,}(Q;\square)$ for $M \in \mathbb{N}$, the calculations even show that $\lim_{M \to \infty} u_M = u$ with respect to $\| \cdot \|_{H^1_{0;0,}(Q;\square)}$ and thus, we have that $u \in H^1_{0;0,}(Q;\square)$.
        
    However, $\partial_{tt} u \not \in L^2(Q)$, as
    \begin{equation*}
        \| \partial_{tt} v_j \|_{L^2(0,T)}^2 =  \frac{T^3 \mu _j}{6}+\frac{1}{4} T^2 \sqrt{\mu _j} \sin \left(2 T \sqrt{\mu _j}\right) 
        -\frac{\sin \left(2 T \sqrt{\mu _j}\right)}{8 \sqrt{\mu _j}}-\frac{1}{4} T \cos \left(2 T \sqrt{\mu _j}\right)+\frac{T}{2}
    \end{equation*}
    and
    \begin{equation*}
        \|  \partial_{tt} u  \|_{L^2(Q)}^2 = \sum_{j=1}^\infty \mu_j^{-d/2-1} \| \partial_{tt} v_j \|_{L^2(0,T)}^2 \geq C \sum_{j=1}^\infty j^{-1-2/d+2/d} = C \sum_{j=1}^\infty j^{-1} = \infty
    \end{equation*}
     with $C>0$ depending on $d$, $T$ and $\Omega$, using Weyl’s asymptotic formula~\eqref{zank:VF:Weyl}.
\end{proof}

%------------------------------------------------------------------------------
\section{Conclusions} \label{zank:sec:Zum}

In this work, we derived a variational setting for the wave equation in the second-order formulation with an isomorphism as solution operator for source terms in $L^2(Q)$. The derived solution space is contained in $H^1(Q)$ and is not a subspace of $H^2(0,T;L^2(\Omega))$. We emphasize that this existence and uniqueness result for the wave equation closes open gaps in the literature, see Theorem~\ref{zank:einf:thm:InfSupH1}.
Moreover, this variational setting is useful for the derivation of the boundary integral equations for the wave equation and regularity results for the wave equation as a counterpart to the variational formulation in \cite{SteinbachZank2022VF}. Further, the proposed variational formulation could serve as the basis for various approximation methods, such as wavelets, least-squares approaches, or boundary element methods. Last, conforming discretizations of the proposed variational formulations are also of interest and are left for future work.

%******************************************************************************
\section*{Acknowledgments}
\noindent
This research was funded in part by the Austrian Science Fund (FWF) [10.55776/ESP1999624]. For open access purposes, the author has applied a CC BY public copyright license to any author accepted manuscript version arising from this submission.

%******************************************************************************

\end{document}